\documentclass[10pt]{article}
\usepackage{amsmath}
\usepackage{amssymb}
\usepackage{amsthm}
\usepackage{lineno}
\usepackage[mathscr]{eucal}
\usepackage{xcolor}
\theoremstyle{definition}
\newtheorem{definition}{Definition}[section]
\newtheorem{theorem}[definition]{Theorem}
\newtheorem*{theorem*}{Conjecture}
\newtheorem{proposition}[definition]{Proposition}

\theoremstyle{remark}
\newtheorem{remark}[definition]{Remark}

\newcounter{enumctr}

\newcommand{\N}{\mathbb{N}}
\newcommand{\R}{\mathbb{R}}

\begin{document}


\title{\vspace*{-10mm}
On the existence and uniqueness of weak solutions to time-fractional elliptic equations with time dependent variable coefficients}
\author{H.T.~Tuan\footnote{\tt httuan@math.ac.vn, \rm Institute of Mathematics, Vietnam Academy of Science and Technology, 18 Hoang Quoc Viet, 10307 Ha Noi, Viet Nam}}
\maketitle
\begin{abstract}
This paper is devoted to discussing the existence and uniqueness of weak solutions to time-fractional elliptic equations having time-dependent variable coefficients. To obtain the main result, our strategy is to combine the Galerkin method, a basic inequality for the fractional derivative of convex Lyapunov candidate functions, the Yoshida approximation sequence and the weak compactness argument.
\end{abstract}
{{\emph {Keywords and phases}}: Time fractional derivatives, Existence and uniqueness, Weak solution, Elliptic equations, Variable coefficients}
\section{Introduction}
Diffusion equations with fractional-order derivatives in time (which is called as time-fractional diffusion equations) have been introduced in Physics by Nigmatullin \cite{Nigmatullin_86} to describe super slow diffusion process in a porous medium with the structure type of fractal geometry (Koch's tree). Then, by the probabilistic point of view, in the paper \cite{Metzler_00}, Metzler and Klafter have pointed out that a time-fractional diffusion equation generates a non-Markovian diffusion process with a long memory. After that, Roman and Alemany \cite{Roman_94} have considered continuous-time random walks on fractals and observed that the average probability density of random walks on fractals obeys a diffusion equation with a fractional time derivative asymptotically. Another context where such systems appear is the modelling of evolution processes in materials with memory, see e.g., \cite{Pruss, Caputo_99}.

The existence of solutions to time-fractional partial differential equations has been studied by many authors. In \cite{Kochubei_04}, using Fourier transform, the authors have obtained a fundamental solution for time-fractional elliptic equations with smooth coefficient. Combining the Galerkin method and the Yoshida approximation sequence, in \cite{Zacher_09},  the author has proposed a way to prove existence of certain weak solutions to abstract evolutionary integro-differential equations in Hilbert spaces. By virtue of the operator theory in functional analysis and  the eigenfunction expansion method for symmetry elliptic operators, in \cite{Yamamoto_11}, Sakamoto and Yamamoto have proved  the existence and uniqueness of the weak solution for a fractional diffusion-wave equation. Based on a definition of the Caputo derivative on a finite interval in fractional Sobolev spaces, Gorenflo, Luchko and Yamamoto  \cite{Gorenflo_15} have investigated solutions (in the distribution sense) to time-fractional diffusion equations from the operator theoretic viewpoint. In the recent works, by combining a special approximate solution with the Galerkin approximation method, Kubica and Yamamoto \cite{Adam} have proved the existence of weak solutions to a large class of fractional diffusion equations with coefficients are dependent on spatial and time variables, by a classical variational approach,  K.V. Bockstal \cite{Bockstal, Bockstal_1} has established the existence of a unique weak solution to a certain class of fractional diffusion equations with Caputo derivative.

However, to the best of our knowledge, the development of this theory is still in its infancy and requires further researches. The main difficulty which one have to face is the meaning of the initial condition of solutions and the correctness of the formulation of weak solutions. These are due to the fact that the solutions of time-fractional partial differential equations with with time-dependent coefficients are not in standard Sobolev spaces and so classical embedding theorems in functional analysis do not apply. In this paper, we only focus on an initial-boundary value problem with the zero initial condition for time-fractional elliptic equations having time dependent variable coefficients which can not apply Fourier transformation to solve them. To overcome the aforementioned obstacles, our strategy is to use the Galerkin method, a basic inequality for the fractional derivative of convex Lyapunov candidate functions, the Yoshida approximation sequence and the weak compactness argument.

The paper is organized as follows. In section 2, we recall some preliminary results on fractional calculus. Then, we give the setting of the problem and propose a clear definition of a weak solution to a time fractional elliptic equation with time dependent variable coefficients. The main result of the paper is Theorem 3.4 on the existence and uniqueness of weak solutions introduced in section 3.

To conclude the introduction, we will introduce some notations used throughout the rest of the article. Denote by $\N$ the set of natural numbers and by $\R$ the set of real numbers. For any natural $d\in \N$, let $\R^d$ be the $d$-dimensional Euclidean space. For a open subset $\Omega$ of $\R^d$, let $C^\infty_c(\Omega)$ denote the space of infinitely differentiable functions $f: \Omega\rightarrow \R$ with the compact support in $\Omega$, $L^p(\Omega)$, $p\in \N$, be the set of measurable functions $f:\Omega\to\R$ such that
\[
\int_\Omega |f(x)|^pdx<\infty,
\]
$H^1(\Omega)$ be the Sobolev space containing all locally integrable functions $f:\Omega\to \R$ such that $f$ and its weak derivatives belong to $L^2(\Omega)$, $H^1_0(\Omega)$ be the closure of $C^\infty_c(\Omega)$ in $H^1(\Omega)$, and $H^{-1}(\Omega)$ be the dual space of $H^1_0(\Omega)$. Fix $T>0$, denote by $W^1_1([0,T];\R)$ the space of functions $f:[0,T];\R$ such that $f$ and its weak derivative belong to the space $L^1([0,T];\R)$ and $W^1_2([0,T];H^{-1}(\Omega))$ by the space of functions $f:[0,T]\rightarrow H^{-1}(\Omega)$ such that $f$ and the weak derivative belong to the space $L^2([0,T];H^{-1}(\Omega))$.
\section{Fractional calculus}
We briefly recall an abstract framework of fractional calculus.

Let $\alpha\in (0,1]$, $[0,T]\subset \R$ and $x:[0,T]\rightarrow \R$ be a measurable function such that $\int_0^T|x(\tau)|\;d\tau<\infty$. The right-handed \emph{Riemann--Liouville integral operator of order $\alpha$} is defined by
\[
(I_{0+}^{\alpha}x)(t):=\frac{1}{\Gamma(\alpha)}\int_0^t(t-\tau)^{\alpha-1}x(\tau)\;d\tau,
\]
where $\Gamma(\cdot)$ is the Gamma function. The left-handed \emph{Riemann--Liouville integral operator of order $\alpha$} is defined by
\[
(I_{T^-}^{\alpha}x)(t):=\frac{1}{\Gamma(\alpha)}\int_t^T(\tau-t)^{\alpha-1}x(\tau)\;d\tau.
\]
We have the following result on the relation between the right-handed and left-handed Riemann--Liouville integral operators.
\begin{theorem}[The Hardy--Littlewood form of fractional integration by parts]
If $p>1$, $q>1$, $0<\alpha<1$, $\frac{1}{p}+\frac{1}{q}-1\leq \alpha$, and \[ f\in L^p([0,T];\R),\;g\in L^q([0,T];\R),
\]
then
\[
\int_0^T (I^\alpha_{0+} f(t))g(t)dt=\int_0^T f(t)I^\alpha_{T^-}g(t)dt. 
\]
\end{theorem}
\begin{proof}
See \cite{LY_37}.
\end{proof}
The right-handed \emph{Riemann--Liouville fractional derivative} $^{RL\!}D_{a+}^\alpha x$ of $x$ on $[0,T]$ is defined by
\[
^{RL\!}D_{0+}^\alpha x(t):=(DI_{0+}^{1-\alpha}x)(t)\;\text{for almost}\;t\in [0,T],
\]
where $D=\frac{d}{dt}$ is the usual derivative. The left-handed \emph{Riemann--Liouville fractional derivative} $^{RL\!}D_{T^-}^\alpha x$ of $x$ on $[0,T]$ is defined by
\[
^{RL\!}D_{T^-}^\alpha x(t):=-(DI_{T^-}^{1-\alpha}x)(t)\;\text{for almost}\;t\in [0,T].
\]
\begin{theorem}\label{integration_by_parts}
The formula
\[
\int_0^T f(t)^{RL}D^\alpha_{0+}g(t)dt=\int_0^T g(t)^{RL}D^\alpha_{T^-}f(t)dt
\]
is valid for $0<\alpha<1$, $f\in I^\alpha_{T^-}(L^p)$, $g\in I^\alpha_{0+}(L^q)$ and $\frac{1}{p}+\frac{1}{q}-1\leq \alpha$.
\end{theorem}
\begin{proof}
See \cite[Corollary 2, pp. 46]{Samko}.
\end{proof}
The right-handed \emph{Caputo fractional derivative} of $x$ on $[0,T]$ is defined by
\[
^{\!C}D^\alpha_{0+}x(t)=^{RL\!}D_{a+}^\alpha (x(t)-x(0))\quad\text{for almost}\;t\in [0,T]
\]
and the left-handed \emph{Caputo fractional derivative} of $x$ on $[0,T]$ is defined by
\[
^{\!C}D^\alpha_{T^-}x(t)=^{RL\!}D_{T^-}^\alpha (x(t)-x(T))\quad\text{for almost}\;t\in [0,T].
\]
We have a sufficient condition for the existence of fractional derivative.
\begin{theorem}
Let $f\in AC([0,T];\R)$, $\alpha\in (0,1)$, then $^{RL}D^\alpha_{0+}f$ and $^{RL}D^\alpha_{T^-}f$ exist almost everywhere. Moreover, $^{RL}D^\alpha_{0+}f, ^{RL}D^\alpha_{T^-}f\in L^r([0,T];\R)$, $1\leq r\leq \frac{1}{\alpha}$, and
\[
^{RL}D^\alpha_{0+}f(t)=\frac{1}{\Gamma(1-\alpha)}\Big[\frac{f(0)}{t^\alpha}+\int_0^t \frac{f'(\tau)}{(t-\tau)^\alpha}d\tau\Big],
\]
\[
^{RL}D^\alpha_{T^-}f(t)=\frac{1}{\Gamma(1-\alpha)}\Big[\frac{f(T)}{(T-t)^\alpha}-\int_t^T \frac{f'(\tau)}{(\tau-t)^\alpha}d\tau\Big].
\]
\end{theorem}
\begin{proof}
See \cite[Lemma 2.2, pp. 35--36]{Samko}.
\end{proof}
\begin{definition}\label{w_sol}
Let $u\in L^1([0,T];H^1_0(\Omega))$. We define the weak Riemann--Liouville fractional derivative of the order $\alpha$ of $u$, $^{RL}D^\alpha_{0+}u(t)$, as below
\[
\int_0^T \varphi(t){^{RL}D^\alpha_{0+}}u(t)dt=\int_0^T{^{RL}D^\alpha_{T^-}}\varphi(t)u(t)dt
\]
for all $\varphi\in C^\infty_c((0,T);\R)$.
\end{definition}
%
%
\section{Weak solutions to time-fractional elliptic equations}
Let $\Omega$ be a bounded domain in $\R^d$ with the boundary $\partial \Omega\in C^1$, $T>0$ and $\alpha\in (0,1)$. Denote $\Omega_T=(0,T]\times\Omega$. We consider the equation of the order $\alpha$
\begin{align}\label{main_eq}
\frac{\partial^{\alpha} u(t,x)}{\partial t^\alpha}-\sum_{i,j=1}^d \partial_{x_i}(a_{ij}(t,x)\partial_{x_j} u(t,x))+\sum_{j=1}^d b_j(t,x)\partial_{x_j}u(t,x)+c(t,x)u(t,x)=f(t)
\end{align}
for $(t,x)\in \Omega_T,$ where $\frac{\partial^{\alpha} u(t,\cdot)}{\partial t^\alpha}$ is the weak Riemann--Liouville fractional derivative of the order $\alpha$ of $u$ with respect to the time variable $t$ and
\begin{itemize}
\item[(a1)] $a_{ij}, b_j, c\in L^\infty (\Omega_T;\R)$ for all $1\leq i,j\leq d$;
\item[(a2)] $a_{ij}=a_{ji}$ for all $1\leq i,j\leq d$;
\item[(a3)] there exists $\theta>0$ such that $\sum_{i,j=1}^d a_{ij}(t,x)\xi_i\xi_j\geq \theta \|\xi\|^2$ for a.e. $t\in (0,T)$, $x\in \Omega$ and for all $\xi\in \R^d$;
\item[(a4)] $f\in L^\infty([0,T];H^{-1}(\Omega))$. 
\end{itemize}
Assume that
\begin{equation}\label{boundary_eq}
u(t,x)=0\; \text{on}\;[0,T]\time\partial\times\partial \Omega.
\end{equation}
Denote 
\[
a(u,v;t):=\int_{\Omega}\sum_{1\leq i,j\leq d}a_{ij}\partial_{x_j}u\partial_{x_i}v+\sum_{1\leq j\leq d}b_j\partial_{x_j}u v+cuv
\]
for all $u,v\in H^1_0(\Omega)$ and $\bold{u}(t):=[u(t)](x)$ for all $t\in [0,T]$, $x\in \Omega$.
\begin{definition}[Weak solution]
A function $\bold{u}:[0,T]\rightarrow H^1_0(\Omega)$ is a weak solution to the problem \eqref{main_eq} with the condition \eqref{boundary_eq} if 
\begin{itemize}
\item[(i)] $\bold{u}\in L^2([0,T];H^1_0(\Omega))$ and $^{RL}D^\alpha_{0+}\bold{u}\in L^2([0,T];H^{-1}(\Omega))$;
\item[(ii)] for all $v\in H^1_0(\Omega)$
\[
\langle ^{RL}D^\alpha_{0+}\bold{u}(t),v\rangle_{H^{-1}\times H^1_0}+a(\bold{u}(t),v;t)=\langle f(t),v\rangle_{H^{-1}\times H^1_0}
\]
for a.e. $t\in (0,T)$, where $\langle \cdot,\cdot\rangle_{H^{-1}\times H^1_0}$ is the duality pairing between $H^{-1}(\Omega)$ and $H^1_0(\Omega)$.
\end{itemize}
\end{definition}
\subsection{Galerkin approximation solution}
Let $\{e_j\}_{j=1}^\infty$ be smooth functions which constitutes an orthonormal basic of $L^2(\Omega)$ and a basis of $H^1_0(\Omega)$ such that 
\begin{itemize}
\item[(i)] $-\Delta e_k=\lambda_k e_k$, $k\in \N$;
\item[(ii)] \[
\int_\Omega e_j e_j=\begin{cases}
1,\quad i=j\\
0,\quad i\ne j
\end{cases}
\text{and}\;
\int_{\Omega}De_i De_j =\begin{cases}
\lambda_i,\quad i=j\\
0,\quad i\ne j
\end{cases}
\]
\end{itemize}
For the existence of these functions, see \cite[p. 335]{Evans_98}. Fix $N\in \N$. Let $E_N=\text{span}\{e_1,\dots,e_N\}$ and 
$P_N$ be the project map from $L^2(\Omega)$ to $E_N$ defined by
$P_Nu=\sum_{i=1}^Nc^ie_i$ with $u\in L^2(\Omega)$ has the form $u=\sum_{i=1}^\infty c^i e_i$ in which $c^i=\int_{\Omega}ue_idx$, $i\in\N$. In this section, using Galerkin method, we will construct approximation solutions to the problem \eqref{main_eq}--\eqref{boundary_eq} in $E_N$.

Let the function $\bold{u}_N:[0,T]\rightarrow E_N$ having the form
\begin{equation}\label{Galerkin_sol}
\bold{u}_N(t)=\sum_{i=1}^N c^i(t)e_i,\quad t\in [0,T],
\end{equation}
here $c^i(\cdot)$, $1\leq i\leq N$, is continuous and has the Riemann--Liouville fractional derivative of the order $\alpha$ on $[0,T]$. Assume that 
\begin{equation}\label{gas}
(^{RL}D^\alpha_{0+}\bold{u}_{N},v)_{L^2}+a(\bold{u}_N(t),v;t)=\langle f(t),v\rangle_{H^{-1}\times H^1_0}
\end{equation}
for a.e. $t\in (0,T)$ and 
\begin{equation}\label{ini_cond}
\bold{u}_N(0)=0.
\end{equation}
We obtain the following result.
\begin{proposition}
For any $N\in \N$, there exists a unique solution to the problem \eqref{gas}--\eqref{ini_cond} having the form \eqref{Galerkin_sol}.
\end{proposition}
\begin{proof}
Consider the function $\bold{u}_N:[0,T]\rightarrow E_N$ having the form
\[
\bold{u}_N(t)=\sum_{i=1}^N c^i(t)e_i,\quad t\in [0,T],
\]
where $c^i(\cdot)$, $1\leq i\leq N$, is continuous and has the Riemann--Liouville fractional derivative of the order $\alpha$ on $[0,T]$. To $\bold{u}_N(\cdot)$ is a solution to \eqref{gas}--\eqref{ini_cond} then $c^i(\cdot)$, $1\leq i\leq N$, have to satisfy the following condition
\begin{align}
&(\sum_{i=1}^N {^{RL}D^\alpha_{0+}c^{i}(t)}e_i,e_j)_{L^2}+a(\sum_{i=1}^N c^i(t)e_i,e_j;t)\label{e1}\\
\notag&\hspace{1cm}={^{RL}D^\alpha_{0+}}c^{j}(t)+\sum_{i=1}^N c^i(t)a(e_i,e_j;t)=\langle f(t),e_j\rangle_{H^{-1}\times H^1_0}=f^j(t)
\end{align}
for $1\leq j\leq N$ and almost every $t\in (0,T]$. Moreover,
\begin{equation}\label{e2}
c^j(0)=0,\quad 1\leq j\leq N.
\end{equation}
Put $\overrightarrow{c}(t)=(c^1(t),\dots,c^N(t))^{\rm T}$, $A(t)=(a(e_i,e_j;t))_{1\leq i,j\leq N}$, and $$\overrightarrow{f}(t)=(f^1(t),\dots,f^N(t))^{\rm T}.$$ Then the system \eqref{e1}--\eqref{e2} is rewritten in the form
\begin{align}
\label{e3}^{RL}D^\alpha_{0+}\overrightarrow{c}(t)+A(t)\overrightarrow{c}(t)={^{C}D^\alpha_{0+}}\overrightarrow{c}(t)+A(t)\overrightarrow{c}(t)&=\overrightarrow{f}(t),\quad t\in (0,T],\\
\label{e4}\overrightarrow{c}(0)&=0.
\end{align}
Hence, the system \eqref{e3}--\eqref{e4} has a solution on $[0,T]$ if and only if the following integral equation has a continuous solution
\begin{equation}\label{equi}
\overrightarrow{c}(t)=-\frac{1}{\Gamma(\alpha)}\int_0^t (t-\tau)^{\alpha-1}A(\tau)\overrightarrow{c}(\tau)d\tau+\frac{1}{\Gamma(\alpha)}\int_0^t (t-\tau)^{\alpha-1}\overrightarrow{f}(\tau)d\tau,\;t\in [0,T].
\end{equation}
On the space $C([0,T];\R^N)$, we establish an operator as
\[
\mathcal{T}_0\varphi(t)=-\frac{1}{\Gamma(\alpha)}\int_0^t (t-\tau)^{\alpha-1}A(\tau)\varphi(\tau)d\tau+\frac{1}{\Gamma(\alpha)}\int_0^t (t-\tau)^{\alpha-1}\overrightarrow{f}(\tau)d\tau,\;t\in [0,T].
\]
For any $\gamma>0$, define a norm $\|\cdot\|_\gamma$ on $C([0,T];\R^N)$ by
\[
\|\varphi\|_\gamma:=\max_{t\in [0,T]}\frac{\|\varphi(t)\|}{\exp{(\gamma t)}}.
\]
It is obvious that $(C([0,T];\R^N), \|\cdot\|_\gamma)$ is a Banach space. On the other hand, for any $\varphi,\tilde{\varphi}\in C([0,T];\R^N)$ and $t\in [0,T]$, we have
\begin{align*}
\frac{\|\mathcal{T}_0\varphi(t)-\mathcal{T}_0\tilde{\varphi}(t)\|}{\exp{(\gamma t)}}&\leq \frac{\text{ess sup}_{t\in [0,T]}\|A(t)\|}{\Gamma(\alpha)}\int_0^t (t-\tau)^{\alpha-1}\exp{(-\gamma(t-\tau))}\frac{\|\varphi(\tau)-\tilde{\varphi}(\tau)\|}{\exp{(\gamma \tau)}}d\tau\\
&\leq \frac{\text{ess sup}_{t\in [0,T]}\|A(t)\|}{\gamma^\alpha\Gamma(\alpha)}\int_0^{\gamma t} u^{\alpha-1}\exp{(-u)}du\|\varphi-\tilde{\varphi}\|_\gamma\\
&\leq\frac{\text{ess sup}_{t\in [0,T]}\|A(t)\|}{\gamma^\alpha}\|\varphi-\tilde{\varphi}\|_\gamma.
\end{align*}
Hence, for $\gamma>0$ large enough, the operator $\mathcal{T}_0$ is contractive in $(C([0,T];\R^N), \|\cdot\|_\gamma)$ and has a unique fixed point which is also the solution to the system \eqref{equi}. The proof is complete.
\end{proof}
We now give some estimate of the Galerkin approximation solution.
\begin{proposition}
There exists a positive constant $C$, depending on $\Omega$, $T$ and the coefficients of the equation \eqref{main_eq}, such that for all $N\in \N$ it holds that
\begin{align*}
&\|\bold{u}_N\|^2_{L^2([0,T];L^2(\Omega))}+\|\bold{u}_N\|^2_{L^2([0,T];H^1_0(\Omega))}+\|^{RL}D^\alpha_{0+}\bold{u}_N\|^2_{L^2([0,T];H^{-1}(\Omega))}\\
&\hspace{2cm}\leq C\|f\|^2_{L^\infty([0,T];H^{-1}(\Omega)}.
\end{align*}
\end{proposition}
\begin{proof}
First, using \cite[Theorem 2]{Tuan_18}, we have
\begin{align}
\notag \langle^{RL}D^\alpha_{0+}\bold{u}_N(t),\bold{u}_N(t)\rangle_{H^{-1}\times H^1_0}&=\langle^{C}D^\alpha_{0+}\bold{u}_N(t),\bold{u}_N(t)\rangle_{H^{-1}\times H^1_0}\\
\notag&=\Big(\sum_{i=1}^N{^{C}D^\alpha_{0+}}c^i(t)e_i,\sum_{i=1}^N c^i(t)e_i\Big)_{L^2}\\
\notag &=\sum_{i=1}^N {c^i(t)^{C}D^\alpha_{0+}}c^i(t)\\
\notag&\geq \frac{1}{2}\sum_{i=1}^N {^{C}D^\alpha_{0+}}(c^i(t))^2\\
\notag&=\frac{1}{2}{^{C}D^\alpha_{0+}}\sum_{i=1}^N (c^i(t))^2\\
\notag&=\frac{1}{2}{^{C}D^\alpha_{0+}}\|\bold{u}_N(t)\|^2_{L^2(\Omega)}.
\end{align}
On the other hand, by \cite[Theorem 3, p. 300]{Evans_98}, there exist $\beta>0$ and $\nu\geq 0$ such that
\begin{equation*}
\beta \|\bold{u}_N(t)\|^2_{H^1_0}\leq a(\bold{u}_N(t),\bold{u}_N(t);t)+\nu \|\bold{u}_N(t)\|^2_{L^2}\; \text{for almost every}\; t\in (0,T).
\end{equation*}
Moreover, from the assumption of $f$ and the Cauchy inequality
\begin{align}
\notag\langle f(t),\bold{u}_N(t)\rangle_{H^{-1}\times H^1_0}&\leq \|f(t)\|_{H^{-1}(\Omega)}\|\bold{u}_N(t)\|_{H^1_0(\Omega)}\\
\notag&\leq \frac{1}{4\beta}\|f(t)\|^2_{H^{-1}(\Omega)}+\beta \|\bold{u}_N(t)\|^2_{H^1_0(\Omega)}.
\end{align}
Thus, for almost every $t\in (0,T)$,
\begin{equation}
\notag ^{C}D^\alpha_{0+}\|\bold{u}_N(t)\|^2_{L^2(\Omega)}\leq 2\nu \|\bold{u}_N(t)\|^2_{L^2(\Omega)}+\frac{1}{2\beta}\|f(t)\|^2_{H^{-1}(\Omega)}.
\end{equation}
Put $v(t):=\|\bold{u}_N(t)\|^2_{L^2}$, $h(t):=\frac{1}{2\beta}\|f(t)\|^2_{H^{-1}}$ and use the comparison principle for solutions to fractional differential equation and the variation of constants formula for solutions to the equations (see \cite[Lemma 3.1]{Tuan_17}), we obtain the estimate
\begin{equation}
\notag v(t)\leq\int_0^t (t-\tau)^{\alpha-1}E_{\alpha,\alpha}(2\nu (t-\tau)^\alpha)h(\tau)d\tau \;\text{for a.a.}\; t\in [0,T].
\end{equation} 
This implies that
\begin{equation}\label{e5}
\|\bold{u}_N(t)\|^2_{L^2(\Omega)}\leq \frac{T^\alpha E_{\alpha,\alpha+1}(2\nu T^\alpha)\|f\|^2_{L^\infty(0,T;H^{-1}(\Omega))}}{2\beta}\;\text{for a.a.}\; t\in [0,T].
\end{equation}
Next, by the similar arguments as above, we see that
\begin{align}
\notag ^{C}D^\alpha_{0+}\|\bold{u}_N(t)\|^2_{L^2(\Omega)}+\beta \|\bold{u}_N(t)\|^2_{H^1_0(\Omega)}&\leq \frac{1}{\beta}\|f(t)\|^2_{H^{-1}(\Omega)}+2\nu \|\bold{u}_N(t)\|^2_{L^2(\Omega)}.
\end{align}
Furthermore,
\begin{align*}
^{C}D^\alpha_{0+}\|\bold{u}_N(t)\|^2_{L^2(\Omega)}={^{RL}D^\alpha_{0+}}\|\bold{u}_N(t)\|^2_{L^2(\Omega)}=\frac{d}{dt}I^{1-\alpha}_{0+}(\|\bold{u}_N(t)\|^2_{L^2(\Omega)}),
\end{align*}
\begin{align}
\notag I^{1-\alpha}_{0+}(\|\bold{u}_N(t)\|^2_{L^2(\Omega)})|_0^T&+\beta\|\bold{u}_N\|^2_{L^2(0,T;H^1_0(\Omega))}\\
\notag&\leq \frac{T\|f\|^2_{L^\infty(0,T;H^{-1}(\Omega))}}{\beta}+2\nu \int_0^T \|\bold{u}_N(t)\|^2_{L^2(\Omega)}dt,
\end{align}
Thus,
\[
\beta\|\bold{u}_N\|^2_{L^2([0,T];H^1_0(\Omega))}\leq \frac{T\|f\|^2_{L^\infty([0,T];H^{-1}(\Omega))}}{\beta}+2\nu \int_0^T \|\bold{u}_N(t)\|^2_{L^2(\Omega)}dt.
\]
This combines with \eqref{e5} implies that there is a constant $C_1>0$ such that
\begin{equation}\label{e6}
\|\bold{u}_N\|^2_{L^2([0,T];H^1_0(\Omega))}\leq C_1 \|f\|^2_{L^\infty ([0,T];H^{-1}(\Omega))}.
\end{equation}
Finally, fix any $v\in H^1_0(\Omega)$ with $\|v\|_{H^1_0}\leq 1$, and write $v=v_0+v_1$, where $v_1\in \text{span}\{e_1,\dots,e_N\}$ and $(e_i,v_0)=0$, $1\leq i\leq N$. Using the estimate concerning the bilinear operator $a(\cdot,\cdot)$ and the facts that $$(^{RL}D^\alpha_{0+}u_N(t),v_1)_{L^2}=\langle ^{RL}D^\alpha_{0+}u_N(t),v\rangle_{H^{-1}\times H^1_0},$$ $(^{RL}D^\alpha_{0+}u_N(t),v_1)_{L^2}+a(u_N(t),v_1;t)=\langle f(t),v_1\rangle_{H^{-1}\times H^1_0}$ and $\|v_1\|_{H^1_0(\Omega)}\leq \|v\|_{H^1_0(\Omega)}\leq 1$, we obtain
\begin{align}
\notag|\langle ^{RL}D^\alpha_{0+}\bold{u}_N(t), v\rangle_{H^{-1}\times H^1_0}|&\leq |a(\bold{u}_N(t),v_1;t)|+|\langle f(t),v_1\rangle_{H^{-1}\times H^1_0}|\\
\notag & \leq C_2\|\bold{u}_N(t)\|_{H^1_0(\Omega)}\|v\|_{H^1_0(\Omega)}+\|f(t)\|_{H^{-1}(\Omega)}\|v\|_{H^1_0(\Omega)}.
\end{align}
Hence,
\begin{equation}\label{e7}
\notag \|^{RL}D^\alpha_{0+}\bold{u}_N(t)\|_{H^{-1}(\Omega)}\leq C_2\|\bold{u}_N(t)\|_{H^1_0(\Omega)}+\|f(t)\|_{H^{-1}(\Omega)}
\end{equation}
which together with \eqref{e7} and the estimate \eqref{e6} implies
\begin{equation}\label{e8}
\|^{RL}D^\alpha_{0+}\bold{u}_N\|^2_{L^2([0,T];H^{-1}(\Omega))}\leq C_3\|f\|^2_{L^\infty([0,T];H^{-1}(\Omega))}.
\end{equation}
From \eqref{e5}, \eqref{e6} and \eqref{e8}, the proof is complete.
\end{proof}
\subsection{Existence and uniqueness of weak solutions}
We now in a position to state the main result of the paper.
\begin{theorem}
Consider the problem \eqref{main_eq}--\eqref{boundary_eq}. Suppose that assumptions (a1)--(a4) hold. Then, this problem has a unique weak solution.
\end{theorem}
\begin{proof}
First, we prove the system \eqref{main_eq}--\eqref{boundary_eq} has at least one weak solution. From \eqref{e6} and \eqref{e8}, by the Banach--Aloaglu theorem, there exist a sequence $\{n_k\}_{k=1}^\infty$ such that
\begin{align}
\bold{u}_{n_k}&\rightharpoonup u\;\text{in}\; L^2([0,T];H^1_0(\Omega)),\label{e9}\\
^{RL}D^\alpha_{0+}\bold{u}_{n_k}&\rightharpoonup v\;\text{in}\;L^2([0,T];H^{-1}(\Omega)).\label{e10}
\end{align}
Let $\varphi\in C^\infty_c((0,T);\R)$ and $\psi\in H^1_0(\Omega)$ be arbitrary. Then,
\begin{align*}
\notag\int_0^T \langle v(t),\phi(t)\psi\rangle_{H^{-1}\times H^1_0}dt&=\lim_{k\to\infty}\int_0^T \langle ^{RL}D^\alpha_{0+}\bold{u}_{n_k}(t),\phi(t)\psi\rangle_{H^{-1}\times H^1_0}dt\\
\notag&=\lim_{k\to\infty}\int_0^T \phi(t)\langle ^{RL}D^\alpha_{0+}\bold{u}_{n_k}(t),\psi\rangle_{H^{-1}\times H^1_0}dt\\
\notag&=\lim_{k\to\infty}\int_0^T {^{RL}D^\alpha_{T-}}\phi(t)\langle \bold{u}_{n_k}(t),\psi\rangle_{H^{-1}\times H^1_0}dt\\
\notag&=\int_0^T {^{RL}D^\alpha_{T-}}\phi(t)\langle \bold{u}(t),\psi\rangle_{H^{-1}\times H^1_0}dt\\
\notag&=\int_0^T\langle {^{RL}D^\alpha_{0+}}\bold{u}(t),\varphi(t)\psi\rangle_{H^{-1}\times H^1_0}dt\\
&=\int_0^T\varphi(t)\langle {^{RL}D^\alpha_{0+}}\bold{u}(t),\psi\rangle_{H^{-1}\times H^1_0}dt,
\end{align*}
which implies
\begin{equation}\label{e_11}
v(t)={^{RL}D^\alpha_{0+}}\bold{u}(t).
\end{equation}
Fix $N,M\in \N$ and $N>M$. For any $\varphi\in C^\infty_0([0,T];\R)$ and $w\in E_M$, we see that
\begin{align*}
\notag\int_0^T \langle {^{RL}D^\alpha_{0+}\bold{u}_N(t)},\varphi(t)w\rangle_{H^{-1}\times H^1_0}dt&=\int_0^T \varphi(t)\langle {^{RL}D^\alpha_{0+}\bold{u}_N(t)},w\rangle_{H^{-1}\times H^1_0}\\
&\rightarrow \int_0^T \varphi(t)\langle {^{RL}D^\alpha_{0+}\bold{u}(t)},w\rangle_{H^{-1}\times H^1_0},
\end{align*}
\begin{align*}
\notag\int_0^Ta(\bold{u}_N(t),\varphi(t)w;t)dt&=\int_0^T\varphi(t)a(\bold{u}_N(t),w;t)dt\\
&\rightarrow\int_0^T\varphi(t)a(\bold{u}(t),w;t)dt,
\end{align*}
and
\begin{equation*}
\int_0^T \langle f(t),\varphi(t)w\rangle_{H^{-1}\times H^1_0}dt=\int_0^T \varphi(t)\langle f(t),w\rangle_{H^{-1}\times H^1_0}dt.
\end{equation*}
Thus,
\begin{equation}\label{e13}
\langle {^{RL}D^\alpha_{0+}\bold{u}(t)},w\rangle_{H^{-1}\times H^1_0}+a(\bold{u}(t),w;t)=\langle f(t),w\rangle_{H^{-1}\times H^1_0}
\end{equation}
for all $w\in E_M$. Moreover, $\bigcup_{M\in \N}E_M$ is dense in $H^1_0(\Omega)$ which implies that \eqref{e13} is also true for all $w\in H^1_0(\Omega)$ and thus the problem \eqref{main_eq}--\eqref{boundary_eq} has a weak solution.

Next, we give a proof for the uniqueness of the weak solution to \eqref{main_eq}--\eqref{boundary_eq}. Suppose that $u_1, u_2$ are two weak solution to this system. Define $u:=u_1-u_2$. Then, $u$ satisfies
\begin{equation*}
\langle (k\ast u)'(t),v\rangle_{H^{-1}\times H^1_0}+a(u(t),v;t)=0,\;v\in H^1_0(\Omega),\;a.a.\;t\in (0,T).
\end{equation*}
Let $v=u(t)$, then for almost every $t\in (0,T)$
\begin{align}
\label{eq_ex_1}\langle^{RL}D^\alpha_{0+}u(t),u(t)\rangle_{H^{-1}\times H^1_0}+a(u(t),u(t);t)&=\langle (k\ast u)'(t),u(t)\rangle_{H^{-1}\times H^1_0}\\
\notag&\hspace{-1cm}+a(u(t),u(t);t)=0,
\end{align}
where $k(t)=\frac{1}{\Gamma(1-\alpha)t^\alpha}$ for $t>0$. Motivated by Rico Zacher \cite[pp. 291--292]{Zacher_08}, we will approximate the operator $\frac{d}{dt}(k\ast u)$ by the sequences $\{\frac{d}{dt}(k_n\ast u)\}_n$, where $k_n(t):=ns(t)=nE_\alpha(-nt^\alpha)$ (note that by using Laplace transform, we see that $s(\cdot)$ is the unique solution to the equation $s(t)+n(l\ast s)(t)=1$, $t>0$, where $l(t)=\frac{t^{\alpha-1}}{\Gamma(\alpha)}$, $t>0$, see \cite[p. 3]{Ke_2020}). We rewrite the equation \eqref{eq_ex_1} as below
\begin{equation}\label{eq_ex_2}
\langle (k_n\ast u)'(t),u(t)\rangle_{H^{-1}\times H^1_0}+a(u(t),u(t);t)=h_n(t),\;a.a.\;t\in (0,T)
\end{equation}
with
\[
h_n(t):=\langle (k_n\ast u)'(t)-(k\ast u)'(t),u(t)\rangle_{H^{-1}\times H^1_0},\;a.a.\;t\in (0,T).
\]
By virtue of \cite[Lemma 2.1]{Zacher_09},
\[
\frac{1}{2}\frac{d}{dt}(k_n\ast\|u(\cdot)\|^2_{L^2(\Omega)})(t)\leq (\frac{d}{dt}(k_n\ast u)(t),u(t))_{L^2(\Omega)},\;a.a.\;t\in (0,T).
\]
On the other hand, there exists $\nu\geq 0$ such that
\[
a(u(t),u(t);t)\geq -\nu\|u(t)\|^2_{L^2(\Omega)}.
\]
Hence, from \eqref{eq_ex_2}, we have
\[
\frac{d}{dt}(k_n\ast\|u(\cdot)\|^2_{L^2(\Omega)})(t)\leq 2\nu\|u(t)\|^2_{L^2(\Omega)}+2h_n(t),\;a.a.\;t\in (0,T).
\]
This together with the positivity of $l$ implies that
\[
l\ast\frac{d}{dt}(k_n\ast \|u(\cdot)\|^2_{L^2(\Omega)})\leq 2\nu(l\ast \|u(\cdot)\|^2_{L^2(\Omega)})(t)+2l\ast h_n(t),\; a.a.\; t\in (0,T).
\]
We will show that 
\begin{equation}\label{key_est}
\lim_{n\to\infty}h_n=0\;\text{in}\; L^1([0,T];\R).
\end{equation}
From the facts that $u\in L^2([0,T];H^{-1}(\Omega))$, $^{RL}D^\alpha_{0+}u=\frac{d}{dt}(k\ast u)\in L^2([0,T];H^{-1}(\Omega))$, \cite[Proposition 2.1, p. 293]{Zacher_08} and  \cite[Example 2.1, p. 294]{Zacher_08}, we have
\[
k\ast \|u(\cdot)\|^2_{H^{-1}(\Omega)}\in W^1_1([0,T];\R).
\]
This implies that
\[
u\in D(B_2),\; \|u(\cdot)\|^2_{H^{-1}(\Omega)}\in D(B_1),
\]
where
\[
B_1(u)=\frac{d}{dt}k\ast u,\;D(B_1)=\{u\in L^1(0,T):k\ast u\in W^1_1([0,T];\R)\},
\]
and 
\[
B_2(u)=\frac{d}{dt}k\ast u,\;D(B_2)=\{u\in L^2([0,T];H^{-1}(\Omega)):k\ast u\in W^1_2([0,T];H^{-1}(\Omega))\}.
\]
Hence, by using \cite[Estimate (18), p. 292]{Zacher_08}, we obtain
\begin{equation}\label{uniq_e}
\lim_{n\to \infty}\int_0^T \|\frac{d}{dt}[(k-k_n)\ast u](t)\|^2_{H^{-1}(\Omega)}dt=\lim_{n\to \infty}\|\frac{d}{dt}[(k-k_n)\ast u](\cdot)\|^2_{L^2([0,T];H^{-1}(\Omega))}=0.
\end{equation}
By \eqref{uniq_e} and the Holder inequality, the following estimates hold
\begin{align*}
\lim_{n\to \infty}\int_0^T|h_n(t)|dt&=\lim_{n\to \infty}\int_0^T|\langle (k_n\ast u)'(t)-(k\ast u)'(t),u(t)\rangle_{H^{-1}\times H^1_0}|dt\\
&\leq \lim_{n\to \infty}\int_0^T \|\frac{d}{dt}[(k-k_n)\ast u](t)\|_{H^{-1}(\Omega)}\|u(t)\|_{H^1_0(\Omega)}\\
&\leq \lim_{n\to \infty} \|\frac{d}{dt}[(k-k_n)\ast u](\cdot)\|_{L^2([0,T];H^{-1}(\Omega))}\|u\|_{L^2([0,T];H^1_0(\Omega))}\\
&=0,
\end{align*}
which shows that $\lim_{n\to\infty}h_n=0$ in $L^1([0,T])$.
Notice that
\begin{align*}
(l\ast k)(t)&=\frac{1}{\Gamma(\alpha)\Gamma(1-\alpha)}\int_0^t s^{\alpha-1}(t-s)^{-\alpha}ds\\
&=\frac{1}{\Gamma(\alpha)\Gamma(1-\alpha)}\int_0^1 u^{-\alpha}(1-u)^{\alpha-1}du\\
&=\frac{1}{\Gamma(\alpha)\Gamma(1-\alpha)}B(\alpha,1-\alpha)\\
&=1,\;\forall t>0,
\end{align*}
\[
l\ast\frac{d}{dt}(k_n\ast \|u(\cdot)\|^2_{L^2(\Omega)})=\frac{d}{dt}(k_n\ast l\ast\|u(\cdot)\|^2_{L^2(\Omega)})\to \frac{d}{dt}(k\ast l\ast\|u(\cdot)\|^2_{L^2(\Omega)})=\|u(\cdot)\|^2_{L^2(\Omega)}
\]
in $L^1([0,T];\R)$ as $n\to \infty$ (see \cite[Estimate (19), p. 292]{Zacher_08}) and $l\ast h_n\to 0$ in $L^1([0,T];\R)$ as $n\to \infty$, we obtain
\begin{equation}\label{f_est}
\|u(t)\|^2_{L^2(\Omega)}\leq 2 \nu(l\ast \|u(\cdot)\|^2_{L^2(\Omega)})(t),\;a.a.\; t\in (0,T).
\end{equation}
From \eqref{f_est}, using a Gronwall type inequality as in \cite[Lemma 7.1.1, p. 188]{Henry}, then $\|u(t)\|^2_{L^2(\Omega)}=0$ a.e. in $(0,T)$, that is, $u_1=u_2$. The proof is complete.
\end{proof}
\begin{remark}
The key point in the proof of the uniqueness of weak solution to the problem \eqref{main_eq}--\eqref{boundary_eq} is to show that  $h_n\to 0$ in $L^1([0,T])$ as $n\to\infty$. The approach proposed in \cite{Zacher_09} can not apply directly to this situation because the operator $B=\frac{d}{dt}(k\ast u)$ with domain $D(B)=\{u\in L^2([0,T];H^1_0(\Omega)):\frac
{d}{dt}k\ast u\in L^2([0,T];H^{-1}(\Omega))\}$ is not m-accretive.
\end{remark}
\begin{remark}
In \cite[Theorem 3.1]{Bockstal}, the author has proved the existence and uniqueness of a weak solution to a class of fractional diffusion equations with Caputo derivative. The right hand side of these equations has the same form as in the equation \eqref{main_eq} in our paper. However, to obtain this result, they need additional following assumptions
\begin{itemize}
\item $\|\partial_t f(t)\|_{H^{-1}(\Omega)}\leq Ct^{-\alpha}$;
\item $b_j=0$ ($1\leq j\leq d$), $\partial_t a_{i,j}\in L^\infty(\overline{\Omega_T})$ ($1\leq i,j\leq d$), and $\partial_t c\in L^\infty (\overline{\Omega_T})$, 
\end{itemize}
see \cite[lemmas 3.3, 3.4]{Bockstal}. In this paper, by another approach, we have studied the existence and uniqueness of a weak solution to the fractional-order equation \eqref{main_eq} (with Riemann--Liouville fractional derivative) without the assumptions as mentioned above.
\end{remark}
\begin{remark}
An important step in proving the existence of the weak solution is to show that the function $v$ in \eqref{e10} is the fractional derivative $^{RL}D^\alpha_{0+}u$ (in the weak sense) of the solution $u$. In this paper, we do that by using Definition \ref{w_sol}. This definition is inspired by the integration by the parts formula for Riemann--Liouville fractional derivatives (Theorem \ref{integration_by_parts}). Unfortunately, this property is generally not true for Caputo fractional derivatives (except for the case where these two fractional derivatives coincide). Therefore, the approach as in the present paper cannot be applied to study the existence of weak solutions to partial differential equations with Caputo fractional derivatives in time.
\end{remark}
\section*{Acknowledgement}
This research is supported by The International Center for Research and Postgraduate Training in  Mathematics--Institute of Mathematics--Vietnam Academy of Science and Technology under the Grant ICRTM01-2020.09. A part of this paper was completed while the author was a postdoc at the Vietnam Institute for Advanced Study in Mathematics (VIASM). He would like to thank VIASM for support and hospitality. The author also would like to thank Nguyen Anh Tu, Ha Duc Thai, and the anonymous referee for the valuable time, constructive suggestions, and interesting comments that have helped him improve the quality and presentation of the paper.


\end{document}